\documentclass[reqno]{amsart}

\usepackage[utf8]{inputenc}
\usepackage{amsthm}
\usepackage{amsmath}
\usepackage{amssymb}
\usepackage{enumitem}
\usepackage[dvips]{graphicx}
\usepackage{color}
\usepackage{hyperref}
\usepackage{url}
\usepackage[left=3.5cm, right=3.5cm, paperheight=11.69in]{geometry}
\usepackage{fancyhdr}
\usepackage{mathrsfs}
\usepackage{stmaryrd}

\newtheorem{lemma}{Lemma}
\newtheorem{proposition}{Proposition}
\newtheorem*{theorem*}{Main Theorem}

\theoremstyle{definition}

\theoremstyle{remark}
\newtheorem{remark}{\textsc{Remark}}

\hypersetup{
    pdftitle={On the commutation of generalized means on probability spaces},
    pdfauthor={Leonetti, Matkowski, and Tringali},
    pdfmenubar=false,
    pdffitwindow=true,
    pdfstartview=FitH,
    colorlinks=true,
    linkcolor=blue,
    citecolor=green,
    urlcolor=cyan
}

\DeclareMathSymbol{\widehatsym}{\mathord}{largesymbols}{"62}

\renewcommand{\rho}{\varrho}

\newcommand\mean[2]{{\mathfrak{F}}_{#1}(#2)}

\newcommand\LLf[2]{\mathfrak{L}^{#1}(#2)}

\providecommand{\AAc}{\mathscr{A}}

\providecommand{\Bb}{\mathcal{H}}

\providecommand{\BBc}{\mathscr{B}}

\providecommand{\CCc}{\mathscr{C}}

\providecommand{\hker}{h}

\providecommand{\HHc}{h}

\providecommand{\Kc}{k}

\providecommand{\LLc}{\mathscr{L}}

\providecommand{\MMc}{\mathscr{M}}

\providecommand{\RRb}{\mathbf{R}}

\providecommand{\vc}{\mathfrak{q}}

\providecommand{\op}{{\rm op}}

\newcommand{\fixed}[2][1]{%
  \begingroup
  \spaceskip=#1\fontdimen2\font minus \fontdimen4\font
  \xspaceskip=0pt\relax
  #2%
  \endgroup
}

\begin{document}
\title{On the commutation of \\ generalized means on probability spaces}

\author{Paolo Leonetti}
\address{Universit\`a ``Luigi Bocconi'' -- via Sarfatti 25, 20136 Milano, Italy}
\email{leonettipaolo@gmail.com}

\author{Janusz Matkowski}
\address{Department of Mathematics, Computer Science and Econometrics, University of Zielona G\'ora --
Podg\'orna 50, PL-65516 Zielona G\'ora, Poland}
\email{j.matkowski@wmie.uz.zgora.pl}
\urladdr{http://januszmatkowski.com/index.html}

\author{Salvatore Tringali}
\address{Texas A\&{}M University at Qatar, Education City -- PO Box 23874, Doha, Qatar}
\email{salvo.tringali@gmail.com}
\urladdr{http://imsc.uni-graz.at/tringali}

%
%

\begin{abstract}
Let $f$ and $g$ be real-valued continuous injections defined on a non-empty real interval $I$, and let $(X, \mathscr{L}, \lambda)$ and $(Y, \mathscr{M}, \mu)$ be probability spaces in each of which there is at least one measurable set whose measure is strictly between $0$ and $1$.

We say that $(f,g)$ is a $(\lambda, \mu)$-switch if, for every $\mathscr{L} \otimes \mathscr{M}$-measurable function $h: X \times Y \to \mathbf{R}$ for which $h[X\times Y]$ is contained in a compact subset of $I$, it holds
$$
f^{-1}\!\left(\int_X f\!\left(g^{-1}\!\left(\int_Y g \circ h\;d\mu\right)\right)d \lambda\right)\! = g^{-1}\!\left(\int_Y g\!\left(f^{-1}\!\left(\int_X f \circ h\;d\lambda\right)\right)d \mu\right)\!,
$$
where $f^{-1}$ is the inverse of the corestriction of $f$ to $f[I]$, and similarly for $g^{-1}$.

We prove that this notion is well-defined, by establishing that the above functional equation is well-posed (the equation can be interpreted as a permutation of generalized means and raised as a problem in the theory of decision making under uncertainty), and show that $(f,g)$ is a $(\lambda, \mu)$-switch if and only if $f = ag + b$ for some $a,b \in \mathbf R$, $a \ne 0$.
\end{abstract}
\maketitle
\thispagestyle{empty}

\section{Introduction}
\label{sec:intro}

Below, we let $I \subseteq \RRb$ be a non-empty interval, which may be bounded or unbounded, and neither open nor closed.
We will need the following proposition, which is proved in Section \ref{sec:well-posedness} (see Section \ref{sec:notations} for a glossary of notation and terms used but not defined in this introduction):

\begin{proposition}\label{th:well-posed}
Let $(S, \CCc, \gamma)$ be a probability space, and assume $w: I \to \RRb$ and $\HHc: S \to I$ are functions such that $w[I]$ is an interval and $w \circ \HHc$
is $\gamma$-integrable. Then $\int_S w \circ \HHc \;d\gamma \in w[I]$.
\end{proposition}

Given $(S, \CCc,\gamma)$ and $w$ as in Proposition \ref{th:well-posed}, we denote by $\LLf{w}{\gamma}$ the set of all $\CCc$-measurable functions $\HHc: S \to I$ such that $w \circ \HHc$ is $\gamma$-integrable,
while we write $\Bb(\gamma)$ for the set of all $\CCc$-measurable functions $\HHc: S \to I$ for which $\HHc[S] \Subset I$.

Based on these premises, assume $w$ is an injection, so that we can consider the inverse, $w^{-1}$, of $w$. It follows from Proposition \ref{th:well-posed} that the functional
\begin{equation}\label{equ:generalized_mean}
\LLf{w}{\gamma} \to \RRb: \HHc \mapsto w^{-1}\!\left(\int_S w \circ \HHc\,d\gamma\right)\!,
\end{equation}
which we denote by $\mean{\gamma}{w}$ and refer to as the \textit{$w$-mean relative to $\gamma$}, is well-defined and its image is contained in $I$. For $\HHc \in \LLf{w}{\gamma}$ we call $\mean{\gamma}{w}(\HHc)$ the $w$-mean of $\HHc$ relative to $\gamma$.

The naming comes from the observation that, if $I$ is the interval $]0,\infty[$ and $w$ is, for some real $p \ne 0$, the function $I \to \RRb: x \mapsto x^p$, then $\LLf{w}{\gamma}$ is the set of all ($\CCc$-measurable and positive) functions $S \to I$ whose $p$-th power is $\gamma$-integrable, while $\mean{\gamma}{w}$ is the integral mean
\begin{equation*}
\LLf{w}{\gamma} \to \RRb: \HHc \mapsto \left(\int_S \HHc^p\,d\gamma\right)^{\frac{1}{p}}.
\end{equation*}
When $S$ is a finite set, \eqref{equ:generalized_mean} gives a generalization of classical and weighted means (say, the arithmetic mean, the quadratic mean, the harmonic mean, and others) first considered, respectively, by A. Kolmogorov and M. Nagumo \cite{Kolmo, Nagu} and B.~de~Finetti and T.~Kitagawa \cite{DeFine, Kita}.

Indeed, our interest in Proposition \ref{th:well-posed} is mainly due to the following result, which also will be proved in Section \ref{sec:well-posedness}.
\begin{proposition}\label{cor:double_mean_(i.ii)}
Let $(U, \AAc, \alpha)$ be a measure space and $(V, \BBc, \beta)$ a probability space, and let $w$ be a continuous injection $I \to \RRb$ and $\HHc$ a function $U \times V \to I$. The following hold:
\begin{enumerate}[label={\rm (\roman{*})}]
\item\label{item:(i)} Let $w \circ \HHc_x$ be $\beta$-integrable for every $x \in U$, where $\HHc_x$ is the map $V \to \RRb: y \mapsto \HHc(x,y)$. Then the function $\varphi: U \to \RRb: x \mapsto \mean{\beta}{w}(\HHc_x)$ is well-defined and $\varphi[U] \subseteq I$. Moreover, if $\HHc$ is $\AAc \otimes \BBc$-measurable and $w \circ \HHc$ is bounded, then $\varphi$ is $\AAc$-measurable.
\item\label{item:(ii)} Suppose that $\HHc[U \times V] \Subset I$, and let $\HHc$ be $\AAc \otimes \BBc$-measurable. Then $\varphi[U] \subseteq I$, and $\varphi$ is $\AAc$-measurable and bounded.
\end{enumerate}
\end{proposition}
Suppose now that $(U, \AAc, \alpha)$ and $(V, \BBc, \beta)$ are both probability spaces, and let $u$ and $v$ be continuous injections $I \to \RRb$. By \cite[Theorem 5.6.5]{Bsh}, $u^{-1}$ and $v^{-1}$ are continuous functions too, and we get by \cite[Theorem 5.3.10]{Bsh} and Propositions \ref{th:well-posed} and \ref{cor:double_mean_(i.ii)}.\ref{item:(ii)} that the functional
$$
\Bb(\alpha \otimes \beta)\to \RRb: \HHc \mapsto u^{-1}\!\left(\int_U u\!\left(v^{-1}\!\left(\int_V v \circ \HHc\;d\beta\right)\right)d \alpha\right)\!,
$$
which we denote by $\mean{\alpha, \beta}{u,v}$ and refer to as the $(u,v)$-mean relative to $\alpha \otimes \beta$, is well-defined. Then, for $h \in \Bb(\alpha \otimes \beta)$ we call $\mean{\alpha, \beta}{u,v}(h)$  the $(u,v)$-mean of $h$ relative to $\alpha \otimes \beta$.
\begin{remark}
\label{rem:factorization}
Let $\HHc \in \Bb(\alpha \otimes \beta)$, and denote by $\HHc_x$, for each $x \in U$, the mapping $V \to \RRb: y \mapsto \HHc(x,y)$. Then, by Proposition \ref{cor:double_mean_(i.ii)}.\ref{item:(i)}, the function
$
\Kc: U \to \RRb: x \mapsto \mean{\beta}{v}(\HHc_x)
$
belongs to $\LLf{u}{\alpha}$, and it is not difficult to verify that $\mean{\alpha, \beta}{u,v}(\HHc) = \mean{\alpha}{u}(\Kc)$. This shows a connection between $\mean{\alpha, \beta}{u,v}$ and $\mean{\alpha}{u}$ and explains, we hope, the terminology.
\end{remark}
With the above in mind, assume for the rest of the section that $(X, \LLc, \lambda)$ and $(Y, \MMc, \mu)$ are probability spaces, and let $f$ and $g$ be continuous injections $I \to \RRb$.
We say that the pair $(f,g)$ is a \textit{$(\lambda, \mu)$-switch} if for all $\hker \in \Bb(\lambda \otimes \mu)$ it holds:
\begin{equation}\label{equ:compact_functional_equ}
\mean{\lambda, \mu}{f,g}(\hker) =
\mean{\mu, \lambda}{g,f}(\hker^\op),
\end{equation}
where $\hker^\op$ is the function $Y \times X \to I: (y,x) \mapsto \hker(x,y)$, or more explicitly (it is just a question of unpacking the relevant definitions):
\begin{equation}\label{equ:explicit_functional_equ}
f^{-1}\!\left(\int_X f\!\left(g^{-1}\!\left(\int_Y g \circ \hker\;d\mu\right)\right)d \lambda\right)\! = g^{-1}\!\left(\int_Y g\!\left(f^{-1}\!\left(\int_X f \circ \hker\;d\lambda\right)\right)d \mu\right);
\end{equation}
note that $\hker \in \Bb(\lambda \otimes \mu)$ if and only if $\hker^\op \in \Bb(\mu \otimes \lambda)$, as is necessary for \eqref{equ:compact_functional_equ} to make sense.

It seems worth observing that if $f$ and $g$ are both equal to the identity function $x \mapsto x$ then \eqref{equ:explicit_functional_equ} boils down to an instance of Fubini's theorem, and the same is true, more in general, whenever $f$ and $g$ are affine functions (see Lemma \ref{lem:sufficiency}).

In particular, $(f,g)$ is called a \textit{discrete $(\lambda,\mu)$-switch} if it is a $(\lambda, \mu)$-switch, $X$ and $Y$ are finite sets, and $\LLc$ and $\MMc$ are the powersets of $X$ and $Y$ (i.e., discrete sigma-algebras), respectively.

It is straightforward from \eqref{equ:explicit_functional_equ} and the definitions (we omit further details) that, if  $X$ and $Y$ are non-empty finite sets, and $(x_i)_{1 \le i \le m}$ is a numbering of $X$ and $(y_i)_{1 \le i \le n}$ a numbering of $Y$, then $(f,g)$ is a discrete $(\lambda,\mu)$-switch if and only if
\begin{equation}\label{equ:discrete_functional_equ}
f^{-1}\Bigg(\sum_{i=1}^m \lambda_i f\Bigg(g^{-1}\Bigg(\sum_{j=1}^n  \mu_j \fixed[0.15]{ \text{ }} g(\xi_{i,j})\Bigg)\Bigg)\Bigg)\! = g^{-1}\Bigg(\sum_{j=1}^n \mu_j \fixed[0.15]{ \text{ }} g\Bigg(f^{-1}\Bigg(\sum_{i=1}^m \lambda_i f(\xi_{i,j})\Bigg)\Bigg)\Bigg)
\end{equation}
for every $m$-by-$n$ matrix $(\xi_{i,j})_{1 \le i \le m, 1 \le j \le n}$ with entries in $I$, where for every $i = 1, \ldots, m$ and $j = 1, \ldots, n$ we set $\lambda_i := \lambda(\{x_i\})$ and $\mu_j := \mu(\{y_j\})$; if $m,n \ge 2$, we may also say that $(f,g)$ is a \textit{$(\lambda_1, \ldots, \lambda_{m-1}; \mu_1, \ldots, \mu_{n-1})$-weighted switch}.

In the present work, \eqref{equ:explicit_functional_equ} and \eqref{equ:discrete_functional_equ} are essentially regarded as functional equations in the unknowns $f$ and $g$, and the main question we address can be loosely phrased as: Is there any nice characterization of $(\lambda, \mu)$-switches? An answer is given by the following result, which is the main contribution of the paper and will be proved in Section \ref{sec:main}:
\begin{theorem*}\label{th:main}
Assume that $\lambda$ and $\mu$ are non-degenerate probability measures. Then $(f,g)$ is a $(\lambda, \mu)$-switch if and only if $f = ag+b$ for some $a, b \in \mathbf R$, $a \ne 0$.
\end{theorem*}
The investigation of functional equations involving generalized means dates back at least to the work of G.~Aumann \cite{Auma} on the so-called ``balancing property'', and it has been the subject of intense research for about eighty years, see, e.g.,
\cite[Chapter III]{Hard},
\cite[Chapter 17]{Acze},
\cite{Kalig}, \cite{Matk01, Matk02},
\cite{Matko}, \cite{Pales}, and references therein.

On the other hand, a ``practical motivation'' for being interested in equation \eqref{equ:compact_functional_equ} comes
from the study of certainty equivalences, a notion first introduced by S.~H.~Chew \cite{Chew} in relation to the theory of expected utility and decision making under uncertainty;
the reader may refer to \cite{Erg} and \cite{Str} for current trends in the area and a survey of the literature on the topic, and to \cite[Section 7.3 and Chapters 15, 17, and 20]{Acze} for further reading.

On top of that,
the study of the functional equation \eqref{equ:compact_functional_equ} fits in the mathematical literature on permutable mappings. The research on the topic essentially started in the 1920s, with G.~Julia's m\'emoire \cite{Julia} and J.~F.~Ritt's subsequent work on permutable rational functions \cite{Ritt}.

The field is still active, particularly due, on the one hand, to a number of open problems and important conjectures in fixed point theory, and on the other to various intersections with the study of dynamical systems, see, e.g., \cite{Basal}, \cite{Daepp},
\cite{McDowe}, and references therein.

\section{Notation and conventions}\label{sec:notations}
Through the paper, the letters $i$, $j$, $m$ and $n$ stand for positive integers (unless otherwise noted), and $\RRb$ is the set of real numbers (endowed with its usual structure of ordered field).

We refer to \cite{Bsh} and \cite{Bogac}, respectively, for basic aspects of real analysis and measure theory (including notation and terms not defined here).
Notably, integration shall be always understood in the sense of Lebesgue, measures will take only non-negative real values, and the only topology considered on [subsets of] $\RRb$ will be the [relative topology induced by the] usual topology.

If $f: X \to Y$ is a function and $S$ is a set, we denote by $f[S]$ the (direct) image of $S$ under $f$, namely the set $ \{f(x): x \in S\} \subseteq Y$, and by $f^{-1}[S]$ the inverse image of $S$ under $f$, namely the set $\{x \in X: f(x) \in S\} \subseteq X$.
If, in addition, $f$ is injective, then we use $f^{-1}$ for the inverse of the function $X \to f[X]: x \mapsto f(x)$, viz. the corestriction of $f$ to $f[X]$, and by an abuse of language we refer to $f^{-1}$ as the inverse of $f$.

Given sigma-algebras $\AAc$ and $\BBc$, a function $w: U \to V$ is $(\AAc, \BBc)$-measurable if $U \in \AAc$, $V \in \BBc$, and for every $B \in \BBc$ there exists $A \in \AAc$ such that $w^{-1}[B \cap V] = A \cap U$; in particular, $w$ is called $\AAc$-measurable if it is $(\AAc, \BBc)$-measurable with $\BBc$ the Borel algebra on $\RRb$.

For $a,b \in \RRb \cup \{-\infty, \infty\}$ we write $[a, b]$ for the closed interval $\{x \in \RRb: a \le x \le b\}$, $[a,b[$ for the semi-open interval $[a,b] \setminus \{b\}$, and $]a,b[$ for the open interval $[a,b] \setminus \{a,b\}$.

We say that a probability measure $\gamma: \CCc \to \RRb$ on a set $S$ is non-degenerate if $\gamma[\CCc] \ne \{0, 1\}$, in which case we refer to the triplet $(S, \CCc, \gamma)$ as a non-degenerate probability space.

For $X,Y \subseteq \RRb$, we write $X \Subset Y$ to mean that $X$ is contained in a compact subset of $Y$.

\section{Proof of Propositions \ref{th:well-posed} and \ref{cor:double_mean_(i.ii)}}
\label{sec:well-posedness}

First a remark. If $I = [a,\infty[$ for some positive $a \in \RRb$ and $w$ is the function $I \to \RRb: x \mapsto a$ then $\int_S w \circ h\,d\gamma = a\fixed[0.2]{\text{ }}\gamma(S)$ for every $\CCc$-measurable function $h: S\to I$, but $a \fixed[0.2]{\text{ }}\gamma(S) \in I$ if and only if $\gamma(S) = 1$. This is the reason for having $S$, and not an arbitrary $E \in \CCc$, as a domain of integration in the statement of Proposition \ref{th:well-posed}.

\begin{proof}[Proof of Proposition \ref{th:well-posed}]
Set $
{\sf m} := \inf w[I]$, ${\sf M} := \sup w[I]$, and $J := \int_S w \circ \HHc \;d\gamma$ (the integral exists by the assumption that $w \circ h$ is $\gamma$-integrable). It follows from the monotonicity of the Lebesgue integral and the fact that $\gamma$ is a probability measure that ${\sf m} \le J \le {\sf M}$.

Consequently, the only cases that we have to consider are when (i) ${\sf m} \ne -\infty$ and ${\sf m} \notin w[I]$, or (ii) ${\sf M} \ne \infty$ and ${\sf M} \notin w[I]$ (otherwise the claim is trivial, since $-\infty < J < \infty$),
and both of these cases can be analyzed by essentially the same type of reasoning. Therefore, we restrict our attention to the latter and prove that $J < {\sf M}$: This will lead to the desired conclusion, since $]{\sf m}, {\sf M}[ \fixed[0.1]{\text{ }} \subseteq w[I] \subseteq
[{\sf m}, {\sf M}[$ by the hypothesis that $w[I]$ is an interval.

To start with, let us define, for each $n$, the sets $I_n := \!\left\{y \in I: w(y) \le {\sf M} - \frac{1}{n}\right\} \subseteq I$ and $S_n := \HHc^{-1}[I_n] \subseteq S$. Since ${\sf M} \notin w[I]$, we have that $(I_n)_{n=1}^\infty$ is a countable covering of $I$, which in turn implies that $(S_n)_{n = 1}^\infty$ is a countable covering of $S$.

On the other hand, $S_n \subseteq S_{n+1}$ for every $n$, since clearly $I_n \subseteq I_{n+1}$.
Therefore, we get from \cite[Proposition 1.5.12]{Bogac} that there must exist an integer $v \ge 1$ such that $\gamma(S_n) > 0$ for all $n \ge v$, as otherwise we would obtain
$1 = \gamma(S) = \lim_{n \to \infty} \gamma(S_n) = 0
$, i.e. a contradiction.

Based on the above, let us define ${\sf M}_v := \sup_{x \in S_v} w \circ \HHc(x)$.
By construction, we have ${\sf M}_v \le {\sf M} - \frac{1}{v} < {\sf M}$,
so the basic properties of integrals entail that
$$
J = \int_{S} w \circ \HHc\;d\gamma = \int_{S_{v}} w \circ \HHc\;d\gamma + \int_{S \setminus S_{v}} w \circ \HHc\;d\gamma \le {\sf M}_v \fixed[0.2]{\text{ }} \gamma(S_{v}) + {\sf M} \fixed[0.2]{\text{ }}
\gamma(S \setminus S_{v}) < {\sf M},
$$
which suffices to complete the proof.
\end{proof}

\begin{proof}[Proof of Proposition \ref{cor:double_mean_(i.ii)}]
\ref{item:(i)} Since $w$ is a continuous function, we have by \cite[Theorem 5.3.10]{Bsh} that $w[I]$ is an interval. Hence, Proposition \ref{th:well-posed} yields that, for every $x \in U$, the image of the function
$$
\psi: U \to \RRb: x \mapsto \int_V w \circ \HHc_x \,d\beta
$$
is contained in $w[I]$, which proves that $\varphi$ is well-defined, and hence  $\varphi[U] \subseteq I$. Thus, assume for the remainder that $\HHc$ is $\AAc \otimes \BBc$-measurable and $w \circ \HHc$ is bounded.

Then, by \cite[Corollary 3.3.3]{Bogac},
$\psi$ is an $\AAc$-measurable function, and since $w^{-1}$ is continuous this is enough to conclude,
in view of \cite[Theorem 2.1.5(i)]{Bogac}, that also $\varphi = w^{-1} \circ \psi$ is $\AAc$-measurable.

\ref{item:(ii)} By point \ref{item:(i)} and our assumptions, we have that $\varphi[U] \subseteq I$, $\varphi$ is $\AAc$-measurable, and there is a compact set $K \subseteq \RRb$ such that $\HHc[U \times V] \subseteq K \subseteq I$. We are left to show that $\varphi$ is bounded.

To this end, we note that the continuity of $w$, together with the above considerations, yields that
${\sf m} \le w(\HHc(x,y)) \le {\sf M}$ for all $ (x,y) \in U \times V$,
where ${\sf m}$ and ${\sf M}$ are, respectively, the minimum and maximum of $w$ over $K$, which exist by
Weierstrass' (extreme value) theorem. Thus, we find from basic properties of integrals (cf. the proof of Proposition \ref{th:well-posed}) that
\begin{equation*}
{\sf m} \le \int_V w \circ \HHc_x\,d\beta \le {\sf M} \ \text{ for all } x \in U.
\end{equation*}
On the other hand, $w^{-1}$ is a continuous function $w[I] \to I$ and $J := [{\sf m},{\sf M}] \Subset w[I]$. So we get by \cite[Theorem 5.3.10]{Bsh} and another application of Weierstrass' theorem that $\varphi[U] \subseteq w^{-1}[J] \Subset I$.
\end{proof}

\section{Proof of the Main Theorem}
\label{sec:main}
We split the proof into a series of three lemmas (recall that we are assuming $f$ and $g$ are continuous injections $I \to \RRb$).

We begin with the ``if'' part of the theorem, for which we first need the following elementary proposition.
Throughout, $t_{\alpha,\fixed[0.2]{\text{ }}\beta}$ denotes, for $\alpha,\beta \in \RRb$, the affine function $\RRb \to \RRb: x \mapsto \alpha \fixed[0.2]{\text{ }} x+\beta$.
\begin{proposition}
\label{prop:invariance_of_means}
Let $(S, \CCc, \gamma)$ be a probability space and $w: I \to \RRb$ a [continuous] injection, and fix $\alpha, \beta \in \RRb$, $\alpha \ne 0$. Then,
$t_{\alpha,\fixed[0.2]{\text{ }}\beta} \circ w$ is a [continuous] injection, and $
\mean{\gamma}{t_{\alpha,\fixed[0.2]{\text{ }}\beta} \circ w} = \mean{\gamma}{w}$.
\end{proposition}
\begin{proof}
Clearly, the function $t_{\alpha,\fixed[0.2]{\text{ }}\beta} \circ w$ is [continuous and] injective, and its inverse is the function $w^{-1} \circ t_{\alpha^{-1}, -\alpha^{-1} \beta}$. So we get from basic properties of integrals that, for all $\HHc \in \LLf{w}{\gamma}$,
\begin{equation*}
\begin{split}
\mean{\gamma}{t_{\alpha,\fixed[0.2]{\text{ }}\beta} \circ w}(\HHc)
                    & = w^{-1}\!\left(\alpha^{-1}\int_S (\alpha \fixed[0.2]{\text{ }} w \circ \HHc + \beta)\,d\gamma - \alpha^{-1}\beta\right) \\
                    & = w^{-1}\!\left(\int_S w \circ \HHc \,d\gamma + \alpha^{-1}\beta \int_S d\gamma - \alpha^{-1}\beta\right)
\end{split}
\end{equation*}
which, together with $\int_S d\gamma = \gamma(S) = 1$, implies  $\mean{\gamma}{t_{\alpha,\fixed[0.2]{\text{ }}\beta} \circ w} = \mean{\gamma}{w}$.
\end{proof}
With this in hand, we can prove the following result, which is straightforward by Remark \ref{rem:factorization}, Proposition \ref{prop:invariance_of_means}, and Fubini's theorem, viz. \cite[Theorem 3.4.4]{Bogac} (we omit further details):

\begin{lemma}\label{lem:sufficiency}
For the pair $(f,g)$ to be a $(\lambda,\mu)$-switch it is sufficient that there exist $a, b \in \mathbf R$, $a \ne 0$ such that $f = ag + b$, and necessary and sufficient that $(af+b, cg+d)$ is a $(\lambda, \mu)$-switch for all $a,b,c,d \in \RRb$ with $ac \ne 0$.
\end{lemma}
Now we show how to reduce Equation \eqref{equ:explicit_functional_equ} to the case where the probability spaces under consideration are discrete (i.e., the sigma-algebras of these spaces are discrete).
\begin{lemma}
\label{lem:reduction_to_(2,2)-discrete}
Let $(f,g)$ be a $(\lambda, \mu)$-switch, and suppose that $\lambda$ and $\mu$ are non-degenerate probabilities measures. Then $(f,g)$ is a discrete $(\lambda^\prime, \mu^\prime)$-switch for some non-degenerate probability measures $\lambda^\prime$ and $\mu^\prime$ on the set $\{0, 1\} \subseteq \RRb$.
\end{lemma}

\begin{proof}
By hypothesis, there exist $A \in \mathscr{L}$ and $B \in \mathscr{M}$ such that $0<\lambda(A)<1$ and $0<\mu(B)<1$. Let $A^c := X \setminus A$ and $B^c := Y \setminus B$ for notational convenience, and for a set $Z \subseteq X \times Y$ denote by $\mathbf{1}_Z$ the characteristic function $X \times Y \to \{0,1\}$ that maps an element $z \in X \times Y$ to $1$ if and only if $z \in Z$. It is clear that all the simple functions $h: X\times Y \to I$ of the form
\begin{displaymath}
a \fixed[0.4]{\text{ }}\mathbf{1}_{A \times B} +
b \fixed[0.4]{\text{ }}\mathbf{1}_{A^c\times B} +
c \fixed[0.4]{\text{ }} \mathbf{1}_{A\times B^c} +
d \fixed[0.4]{\text{ }} \mathbf{1}_{A^c \times B^c},
\end{displaymath}
with $a,b,c,d \in I$, are $\mathcal{L} \otimes \mathscr{M}$-measurable and such that $h[X\times Y]$ is contained in a compact subset of $I$ (in fact, $h[X\times Y]$ is a finite set).
Hence, taking $\mathcal{P}$ equal to the powerset of a $2$-element set $\{0, 1\}$, it is seen that \eqref{equ:discrete_functional_equ} holds with respect to the probability measures $\lambda^\prime, \mu^\prime: \mathcal{P} \to [0,\infty]$ defined by $\lambda^\prime(\{0\}) := \lambda(A)$ and $\mu^\prime(\{0\}) := \mu(B)$.
\end{proof}
Lastly, we solve Equation \eqref{equ:discrete_functional_equ} in the case where $m = n = 2$, which is sufficient for the ``only if'' part of the theorem to be proved, after the reduction implied by Lemma \ref{lem:reduction_to_(2,2)-discrete}.

We will need the following result, which belongs to the folklore, but whose short proof we include here for completeness and the sake of exposition:
\begin{proposition}
\label{prop:k-affinity_implies_Jensen-affinity}
Let $\mathbb V = (V, +, \cdot)$ be a vector space over the real or complex field and $D \subseteq V$ a convex set in $\mathbb V$. Next, fix $\kappa \in {]0,1[}$ and let $\Phi: D \to \RRb$ be a $\kappa$-affine function, that is
\begin{equation}
\label{equ:kappa-affinity}
\forall \vc_1, \vc_2 \in D: \Phi(\kappa \fixed[0.2]{\text{ }}\vc_1 + (1-\kappa)\fixed[0.2]{\text{ }} \vc_2) = \kappa \fixed[0.2]{\text{ }} \Phi(\vc_1) + (1-\kappa)\fixed[0.2]{\text{ }}\Phi(\vc_2).
\end{equation}
Then $\Phi$ is $\frac{1}{2}$-affine (or Jensen-affine), namely \eqref{equ:kappa-affinity} holds with $\kappa = \frac{1}{2}$.
\end{proposition}
\begin{proof}
Pick $\vc_1, \vc_2 \in D$. It is then easily verified that
\begin{equation}
\label{equ:Daroczy-Pales identity}
\forall x,y \in D: \frac{x+y}{2} = {\kappa\! \left(\kappa \fixed[0.2]{\text{ }}\frac{x+y}{2} + (1-\kappa)\fixed[0.2]{\text{ }}x\right)} + {(1-\kappa)\!\left(\kappa \fixed[0.2]{\text{ }} y + (1-\kappa)\frac{x+y}{2}\right)};
\end{equation}
this is called the Dar\'oczy-P\'ales identity, and, to the best of our knowledge, it has gone unnoticed, in spite of the straightforwardness of its proof, until a special case of it was used to prove \cite[Lemma 1]{DarPa}. In particular, \eqref{equ:Daroczy-Pales identity} implies, together with \eqref{equ:kappa-affinity} and the convexity of $D$, that
\begin{equation*}
\begin{split}
\Phi\!\left(\frac{\vc_1 + \vc_2}{2}\right)
    & = \kappa \fixed[0.2]{\text{ }} \Phi\!\left(\kappa \fixed[0.2]{\text{ }}\frac{\vc_1+\vc_2}{2} + (1-\kappa)\fixed[0.2]{\text{ }}\vc_1\right) + (1-\kappa)\fixed[0.2]{\text{ }}\Phi\!\left(\kappa \fixed[0.2]{\text{ }} \vc_2 + (1-\kappa)\frac{\vc_1+\vc_2}{2}\right) \\
    & = (\kappa^2 + (1-\kappa)^2) \fixed[0.2]{\text{ }} \Phi\!\left(\frac{\vc_1+\vc_2}{2}\right) + 2\kappa\fixed[0.2]{\text{ }} (1-\kappa) \frac{\Phi(\vc_1) + \Phi(\vc_2)}{2},
\end{split}
\end{equation*}
which, by the arbitrariness of $\vc_1, \vc_2 \in D$, leads to the desired conclusion.
\end{proof}
\begin{lemma}
\label{lem:3}
Fix $\alpha, \beta \in {]0,1[}$, and let $f$ and $g$ be such that, for all $x,y,z,w \in I$,
\begin{equation}
\label{equ:functional}
\begin{split}
f^{-1} & (\alpha f(g^{-1}(\beta\fixed[0.2]{\text{ }} g(x)+(1-\beta)\fixed[0.2]{\text{ }} g(y))) + (1-\alpha) f(g^{-1}(\beta \fixed[0.2]{\text{ }}g(z)+(1-\beta)\fixed[0.2]{\text{ }}g(w)))) \\
& = g^{-1}(\beta\fixed[0.2]{\text{ }} g(f^{-1}(\alpha f(x)+(1-\alpha)f(z))) + (1-\beta)\fixed[0.2]{\text{ }} g(f^{-1}(\alpha f(y)+(1-\alpha)f(w)))),
\end{split}
\end{equation}
i.e. $(f,g)$ is a $(\alpha\fixed[0.35]{\text{ }}; \beta)$-weighted switch. There then exist $a,b \in \RRb$, $a \ne 0$ such that $f = ag+b$.
\end{lemma}
\begin{proof}
Pick an arbitrary $\xi_0 \in I$. By (the second half of) Lemma \ref{lem:sufficiency}, there is no loss of generality in assuming, as we do, that $f(\xi_0) = g(\xi_0) = 0$. Since the claim is obvious if $I = \{\xi_0\}$, we also suppose for the remainder of this proof that $I \setminus\{\xi_0\}$ is non-empty, and
let $x_0 \in I \setminus \{\xi_0\}$.

By these assumptions and the injectivity of $f$ and $g$,
neither of $f(x_0)$ or $g(x_0)$ is zero, and we can ``normalize'' $f$ and $g$, respectively, to the functions
\begin{equation}\label{equ:normalizing_f}
f_0: I_0 \to \RRb: x \mapsto \frac{f(x)}{f(x_0)}
\end{equation}
and
\begin{equation}\label{equ:normalizing_g}
g_0:  I_0 \to \RRb: x \mapsto \frac{g(x)}{g(x_0)},
\end{equation}
where $I_0 := [x_0, \xi_0] \cup [\xi_0, x_0]$ (notice that we do not know whether $x_0 > \xi_0$ or $x_0 < \xi_0$).

Again by Lemma \ref{lem:sufficiency}, the pair $(f_0,g_0)$ is still a $(\alpha\fixed[0.35]{\text{ }};\beta)$-weighted switch, and we want to show that this implies $f_0 = g_0$.
To start with, note that, by construction and \cite[Theorem 5.3.10]{Bsh},
\begin{equation}
\label{equ:normalized_images}
f_0[I_0] = g_0[I_0] = [0,1].
\end{equation}
Pick $s, t, u, v \in
[0,1]$. Letting $x = g_0^{-1}(s)$, $y = g_0^{-1}(t)$, $z = g_0^{-1}(u)$, and $w = g_0^{-1}(v)$ in \eqref{equ:functional} gives
$$
F(s,t,u,v) = G(s,t,u,v),
$$
where for ease of notation we have put
\begin{equation*}
F(s,t,u,v) := f_0^{-1}(\alpha f_0 \circ g_0^{-1}(\beta s+(1-\beta)t) + (1-\alpha) f_0\circ g_0^{-1}(\beta u+(1-\beta)v))
\end{equation*}
and
\begin{equation*}
\begin{split}
G(s,t,u,v) & := g_0^{-1}(\beta \fixed[0.2]{\text{ }}g_0 \circ f_0^{-1}(\alpha f_0 \circ g_0^{-1}(s) + (1-\alpha) f_0 \circ g_0^{-1}(u)) \\
            & \phantom{:}+ (1-\beta)\fixed[0.2]{\text{ }} g_0 \circ f_0^{-1}(\alpha f_0 \circ g_0^{-1}(t) + (1-\alpha) f_0 \circ g_0^{-1}(v))).
\end{split}
\end{equation*}%
Consequently, we have, for all $s,t,u,v \in [0,1]$, that $g_0 \circ F(s,t,u,v) = g_0 \circ G(s,t,u,v)$,
which, assuming $\varphi := g_0 \circ f_0^{-1}$ (so that $\varphi^{-1} = f_0 \circ {g_0}^{-1}$), is equivalent to saying that
\begin{equation}
\label{equ:towards_convexity}
\begin{split}
\varphi & (\alpha\fixed[0.2]{\text{ }}\varphi^{-1}(\beta\fixed[0.2]{\text{ }} s+(1-\beta)\fixed[0.2]{\text{ }}t) + (1-\alpha)\fixed[0.2]{\text{ }}\varphi^{-1}(\beta\fixed[0.2]{\text{ }} u+(1-\beta)\fixed[0.2]{\text{ }}v)) \\
             & = \beta \fixed[0.2]{\text{ }} \varphi(\alpha \fixed[0.2]{\text{ }}\varphi^{-1}(s) + (1-\alpha)\fixed[0.2]{\text{ }}\varphi^{-1}(u)) + (1-\beta)\fixed[0.2]{\text{ }} \varphi(\alpha \fixed[0.2]{\text{ }}\varphi^{-1}(t) + (1-\alpha)\fixed[0.2]{\text{ }} \varphi^{-1}(v)).
\end{split}
\end{equation}%
Notice that $\varphi$ is, by \eqref{equ:normalized_images}, a continuous bijection $[0,1] \to [0,1]$ for which $\varphi(0) = 0$ and $\varphi(1) = 1$ (here is the reason for having nor\-mal\-ized $f$ to $f_0$ and $g$ to $g_0$ in the way we have done).

On the other hand, if we introduce the function
\begin{equation}
\label{equ:auxiliary_Phi}
\Phi: [0,1]^2 \to [0,1]: (p,q) \mapsto \varphi(\alpha \fixed[0.2]{\text{ }}\varphi^{-1}(p) + (1-\alpha)\fixed[0.2]{\text{ }}\varphi^{-1}(q)),
\end{equation}%
we can rewrite \eqref{equ:towards_convexity} in a more convenient form and find that
\begin{equation*}
\forall (s, t), (u, v) \in [0,1]^2: \Phi(\beta \fixed[0.2]{\text{ }}(s,t) +(1-\beta)\fixed[0.2]{\text{ }}(u,v)) = \beta\fixed[0.2]{\text{ }} \Phi(s,t) + (1-\beta) \fixed[0.2]{\text{ }}\Phi(u,v).
\end{equation*}%
It follows from Proposition \ref{prop:k-affinity_implies_Jensen-affinity} that $\Phi$ is $\frac{1}{2}$-affine, and since $\Phi$ is a
continuous function $[0,1]^2 \to [0,1]$ (here is where we use that $\varphi$ is a continuous bijection $[0,1] \to [0,1]$), this in turn is enough to conclude, by \cite[Theorem 13.2.2]{Kucz}, that there exist some $A,B,C \in \RRb$ such that
$$
\forall u,v \in [0,1]: \Phi(u,v) = Au+Bv+C.
$$
In particular, $(A+B)u+C = \Phi(u,u) = u$ for all $u \in [0,1]$ and ${]0,1[} \ni \varphi(\alpha) = \Phi(1,0) = A+C$, which yields $B = 1-A$, $C = 0$, and $0 < A < 1$.
So, by \eqref{equ:auxiliary_Phi} and the bijectivity of $\varphi$, we see that
$$
\forall u,v \in [0,1]: \varphi(\alpha \fixed[0.2]{\text{ }} u + (1-\alpha)v) = A\fixed[0.2]{\text{ }}\varphi(u)+(1-A)\fixed[0.2]{\text{ }}\varphi(v).
$$
To wit, $\varphi$ is an injective $(\alpha, A)$-affine function $[0,1] \to [0,1]$, which implies by \cite[Theorem K]{MaPy95} that $\varphi$ is an injective affine function $[0,1] \to [0,1]$, viz. there exist $a,b \in \RRb$, $a \ne 0$ such that $\varphi(u) = a\fixed[0.1]{\text{ }}u+b$ for all $u \in [0,1]$.

Taking into account that $\varphi = g_0 \circ f_0^{-1}$, $f_0(\xi_0) = g_0(\xi_0) = 0$, and $f_0(x_0) = g_0(x_0) = 1$, it then follows that $g_0 = f_0$, as desired. But this, together with \eqref{equ:normalizing_f} and \eqref{equ:normalizing_g}, shows that
$$
f(x) = \frac{f(x_0)}{g(x_0)} g(x)
$$
for every $x \in I_0$ and every $x_0 \in I \setminus \{\xi_0\}$. So the quotient $f(x)/g(x)$ is constant for $x \in I \setminus \{\xi_0\}$, and since $f(\xi_0) = g(\xi_0) = 0$, we find that $f = ag$ for some $a \in \RRb$, $a \ne 0$.
\end{proof}

\section{Closing remarks}
\label{sec:closing_remarks}

It seems interesting to try to solve Equation \eqref{equ:discrete_functional_equ}
in the presence of constraints on the entries of the matrix $\Xi = (\xi_{i,j})_{1 \le i \le m, 1 \le j \le n}$ that appear on the left- and right-hand side of the same equation: E.g., we may require that the matrix $\Xi$ has rank $k$ (or $\le k$) for some positive integer $k \le \min(m,n)$, or is square and symmetric (respectively, circulant, triangular, or whatsoever), and for each of these cases we may ask whether or not it is still true that the pair $(f,g)$ is a solution to \eqref{equ:discrete_functional_equ} if and only if $f = ag + b$ for some $a,b \in \mathbf R$, $a \ne 0$.

\section*{Acknowledgments}

P.L. and S.T. are thankful to Veronica Cappelli (Universit\`a ``Luigi Bocconi'', IT) for having attracted their attention to a special case (relative to uniform probability measures) of equation \eqref{equ:discrete_functional_equ}, considered, though not solved, in her MSc thesis in economics.

P.L. was supported by a PhD scholarship from Universit\`a ``Luigi Bocconi'', and S.T. by NPRP grant No. [5-101-1-025] from the Qatar National Research Fund (a member of Qatar Foundation). The statements made herein are solely the responsibility of the authors.

\end{document}